\documentclass[adraft]{eptcs}
\pdfoutput=1
\usepackage[utf8]{inputenc}
\usepackage[T1]{fontenc}
\usepackage[shortlabels]{enumitem}
\usepackage{microtype}
\usepackage{xparse}

\usepackage{amsmath,amssymb,amsthm,mathtools}
\usepackage{bbm}
\usepackage{nccmath} %
\numberwithin{equation}{section}
\newcommand{\define}[1]{\textbf{#1}}

\def\p#1{\mathrel{\ooalign{\hfil$\mapstochar\mkern 5mu$\hfil\cr$#1$}}}
\newcommand{\proTo}{\p\rightarrow}
\newcommand{\xproTo}[1]{\overset{#1}{\p\rightarrow}}
\newcommand{\cat}[1]{\mathsf{#1}}

\DeclareMathOperator{\src}{src}
\DeclareMathOperator{\tgt}{tgt}
\DeclareMathOperator{\id}{id}
\newcommand{\bicat}[1]{\mathbf{#1}}
\newcommand{\Cat}{\bicat{Cat}}
\newcommand{\MonCat}{\bicat{MonCat}}
\newcommand{\MonCatLax}{\MonCat_{\mathrm{lax}}}
\newcommand{\Dbl}{\bicat{Dbl}}
\newcommand{\DblLax}{\Dbl_{\mathrm{lax}}}
\newcommand{\dbl}[1]{\mathbbm{#1}}
\NewDocumentCommand{\Csp}{g}{\mathbbm{C}\mathsf{sp}\IfNoValueTF{#1}{}{(#1)}}
\newcommand{\SCsp}[2]{{}_{#1}\mathbbm{C}\mathsf{sp}(#2)}
\newcommand{\DCsp}[1]{{#1}\mathbbm{C}\mathsf{sp}}
\newcommand{\Span}[1]{\mathbbm{S}\mathsf{pan}(#1)}
\newcommand{\MonDbl}[1]{\mathbbm{B}{#1}} %
\DeclareMathOperator{\Comma}{Comma}
\DeclareMathOperator{\Apex}{Apex}
\DeclareMathOperator{\apex}{apex}
\newcommand{\leftfoot}{\operatorname{ft}_L}
\newcommand{\rightfoot}{\operatorname{ft}_R}

\usepackage{thmtools}
\declaretheorem[within=section]{theorem}
\declaretheorem[sibling=theorem]{proposition}
\declaretheorem[sibling=theorem]{corollary}
\declaretheorem[sibling=theorem]{lemma}
\declaretheorem[sibling=theorem,style=definition]{definition}
\declaretheorem[sibling=theorem,style=definition]{example}

\usepackage{quiver}
\newcommand{\inlineCell}[9]{%
\begin{tikzcd}[ampersand replacement=\&, cramped, sep=small]
  {#1} \& {#2} \\
  {#3} \& {#4}
  \arrow[""{name=0, anchor=center, inner sep=0}, "{#5}", "\shortmid"{marking}, from=1-1, to=1-2]
  \arrow["{#7}"', from=1-1, to=2-1]
  \arrow["{#8}", from=1-2, to=2-2]
  \arrow[""{name=1, anchor=center, inner sep=0}, "{#6}"', "\shortmid"{marking}, from=2-1, to=2-2]
  \arrow["{#9}"{description}, draw=none, from=0, to=1]
\end{tikzcd}}

\title{Structured and Decorated Cospans from the Viewpoint of Double Category Theory}
\author{Evan Patterson
\institute{Topos Institute}
\email{evan@epatters.org}
}
\newcommand{\event}{Applied Category Theory 2023}
\newcommand{\titlerunning}{Structured and Decorated Cospans via Double Category Theory}
\newcommand{\authorrunning}{Evan Patterson}

\hypersetup{breaklinks,bookmarksnumbered,
  pdftitle={\titlerunning},pdfauthor={\authorrunning}}
\usepackage[capitalize,noabbrev,nameinlink]{cleveref}

\usepackage[style=eptcs]{biblatex}
\begin{document}
\maketitle

\begin{abstract}
  Structured and decorated cospans are broadly applicable frameworks for
  building bicategories or double categories of open systems. We streamline and
  generalize these frameworks using central concepts of double category theory.
  We show that, under mild hypotheses, double categories of structured cospans
  are cocartesian (have finite double-categorical coproducts) and are
  equipments. The proofs are simple as they utilize appropriate
  double-categorical universal properties. Maps between double categories of
  structured cospans are studied from the same perspective. We then give a new
  construction of the double category of decorated cospans using the recently
  introduced double Grothendieck construction. Besides its conceptual value,
  this reconstruction leads to a natural generalization of decorated cospans,
  which we illustrate through an example motivated by statistical theories and
  other theories of processes.
\end{abstract}

\section{Introduction}
\label{sec:introduction}

A central theme of applied category theory is the mathematical modeling of open
systems: physical or computational systems that interact with each other along
boundaries or interfaces. Within this tradition, mathematical models of open
systems are most commonly based on spans or cospans, an idea now at least
twenty-five years old \cite{katis1997,rosebrugh2005}. Two general frameworks for
building open systems using cospans have emerged: structured cospans
\cite{fiadeiro2007,baez2020} and decorated cospans \cite{fong2015,baez2022}.
Complementing the mathematical theory, structured cospans have been implemented
in the programming
framework \href{https://github.com/AlgebraicJulia/Catlab.jl}{Catlab.jl} and used
to create software tools for epidemiological modeling based on open Petri nets
\cite{baez2017,libkind2022} and open stock and flow diagrams
\cite{baez2022stockflow}. Structured and decorated cospans are now essential
tools of applied and computational category theory.

The categorical description of open systems based on cospans has evolved over
time. Some early works studied categories of cospans, which compose by taking
pushouts. Because pushouts are defined only up to isomorphism, the morphisms of
these categories must be \emph{isomorphism classes} of cospans. This is
unfaithful to implementation, where one always computes with representatives of
an equivalence class, rather than the equivalence class itself. More
fundamentally, systems generally have morphisms of their own---for example,
Petri nets come with homomorphisms between them---and these are lost if open
systems are taken to be morphisms, rather than objects, of a category.

Both problems are solved by passing from categories to a two-dimensional
categorical structure, of which the best-studied are bicategories. Yet this
presents its own difficulties. In addition to composing along their boundaries,
open systems generally admit a symmetric monoidal product that juxtaposes two of
them ``in parallel.'' One then needs to construct not just a bicategory but a
symmetric monoidal bicategory of open systems. Monoidal bicategories are
inherently complicated because they are properly a three-dimensional categorical
structure (namely, tricategories with one object). It was noticed that rather
than constructing a monoidal bicategory directly, it can be easier to first
construct a monoidal \emph{double category} and then obtain the monoidal
bicategory from the globular cells of the double category
\cite{shulman2010,hansen2019}. But since a double category is at least as good
as a bicategory, one may as well consider double categories of open systems.
That is now what is typically done. In recent work, both structured cospans
\cite{baez2020} and decorated cospans \cite{baez2022} have been assembled into
symmetric monoidal double categories.

The thesis of this paper is that viewing open systems as double categories is
not merely a technical device or a means to constructing bicategories, but a
source of mathematical insights that cannot be obtained at the 1-categorical or
even bicategorical levels. To understand why, consider the philosophy behind the
modern theory of double categories, as developed principally by Grandis and
Par\'e, beginning with an account of double limits and colimits
\cite{grandis1999}, and exposited recently by Grandis \cite{grandis2019}.
Another important expression of this viewpoint is Shulman's theory of equipments
\cite{shulman2008}.

A \define{double category} is succinctly defined as a pseudocategory in
$\Cat$.\footnote{Some authors call this structure a \define{pseudo double
    category} but since all double categories in this paper are pseudo, we
  prefer to omit the adjective. Likewise, our double functors are pseudo by
  default. For complete definitions of these concepts, see \cite{grandis2019}.}
Thus, a double category $\dbl{D}$ consists of a category of objects,
$\dbl{D}_0$; a category of morphisms, $\dbl{D}_1$; source and target functors,
$\src, \tgt: \dbl{D}_1 \rightrightarrows \dbl{D}_0$; and external composition
and identity operations,
$\odot: \dbl{D}_1 \times_{\dbl{D}_0} \dbl{D}_1 \to \dbl{D}_1$ and
$\id: \dbl{D}_0 \to \dbl{D}_1$, which obey the category axioms up to coherent
globular isomorphisms in $\dbl{D}_1$. The objects of the category $\dbl{D}_0$
are called the \define{objects} of the double category $\dbl{D}$, the morphisms
of $\dbl{D}_0$ the \define{arrows} of $\dbl{D}$, the objects of $\dbl{D}_1$ the
\define{proarrows} of $\dbl{D}$, and the morphisms of $\dbl{D}_1$ the
\define{cells} of $\dbl{D}$. In particular, on this definition, the proarrows of
a double category are first and foremost the \emph{objects} of a category, which
happen to have a source and target. In important examples of double categories,
such as those of spans, cospans, relations, matrices, profunctors, and
bimodules, the proarrows are best thought of in precisely this way, as objects
that happen to have a source and target. Crucially, this also applies to double
categories of open systems, which are systems that happen to have boundaries.
Shulman calls such double categories (or rather their underlying bicategories)
``$\mathcal{M}od$-like'' after the bicategory of bimodules between rings
\cite{shulman2008}.

Whether one thinks of proarrows primarily as objects or morphisms may seem a
small matter of perspective, but it gains significance through the modern theory
of double categories, where proarrows play the role of objects in all of the
main concepts, such as natural transformations, limits and colimits, commas,
adjunctions, and the Grothendieck construction. The theory is thus well suited
to describe open systems, including those based on spans and cospans. In this
paper, we study structured and decorated cospans from the viewpoint of double
category theory.\footnote{This paper synthesizes a series of blog posts by the
  author:
  \href{https://topos.site/blog/2022/05/grothendieck-construction-for-double-categories/}%
  {``Grothendieck construction for double categories''}
  (2022),
  \href{https://topos.site/blog/2022/05/decorated-cospans-via-the-grothendieck-construction/}%
  {``Decorated cospans via the Grothendieck construction''}
  (2022),
  \href{https://topos.site/blog/2023/03/structured-cospans-as-a-cocartesian-equipment/}%
  {``Structured cospans as a cocartesian equipment''} (2023).}

After reviewing their structure as a double category, we show that structured
cospans form a cocartesian double category, a statement that is stronger yet
easier to prove than being a symmetric monoidal double category. We also show
that structured cospans are an equipment, so altogether form a cocartesian
equipment (\cref{sec:structured-cospans}). Here we see the advantages of
double-categorical universal properties. We then turn to decorated cospans
(\cref{sec:decorated-cospans}), reconstructing the double category of decorated
cospans as an application of the double Grothendieck construction
\cite{cruttwell2022}. As a byproduct, we also generalize decorated cospans in
several directions, which we illustrate through an example motivated by
categorical statistics.

\section{Structured Cospans as a Cocartesian Equipment}
\label{sec:structured-cospans}

Structured cospans represent open systems as cospans whose feet are restricted
compared with the apex \cite{fiadeiro2007,baez2020}. A simple example is open
graphs with boundaries restricted to be \emph{discrete} graphs. Compared with
other techniques, structured cospans have the advantage of being particularly
easy to use, as the hypotheses for the construction are often easy to check in
examples. However, proofs that the construction itself works are more involved
because the mathematical object being constructed---a symmetric monoidal double
category---is complicated, involving a large number of coherence conditions.
Three different correctness proofs for structured cospans have been given: the
first one by direct but lengthy verification of the axioms \cite{courser2020}
and two later ones by more conceptual routes that however import other
sophisticated concepts. These concepts are pseudocategories in the 2-category of
symmetric monoidal categories \cite{baez2020} and symmetric monoidal
bifibrations \cite{baez2022}.

Such difficulties can be bypassed by viewing structured cospans in a different
light, as forming a \emph{cocartesian} double category, even a cocartesian
equipment. Just as a cartesian or cocartesian category can be given the
structure of a symmetric monoidal category by making a choice of finite products
or coproducts, so can a cartesian or cocartesian double category be given the
structure of a symmetric monoidal double category. It is, however, much easier
to prove cocartesianness than to directly construct the symmetric monoidal
product. This circumstance highlights a recurring tension in category theory:
that between universal properties and algebraic structures. Although algebraic
structure is arguably more flexible, universal properties, when they can be
found, are extremely powerful because many consequences and coherences flow
directly from the defining existence and uniqueness statement, which is often
easy to verify in particular situations. Both cocartesian double categories and
equipments are defined by universal properties, whereas a symmetric monoidal
product is a structure on a double category.

\subsection{Double Category of Structured Cospans}

We begin by reviewing the definition of structured cospans and their structure
as a double category \cite{baez2020}.

\begin{proposition}
  Let $L: \cat{A} \to \cat{X}$ be a functor into a category $\cat{X}$ with
  pushouts. Then there is a double category $\SCsp{L}{\cat{X}}$ that has
  \begin{itemize}[noitemsep]
    \item as objects, the objects of $\cat{A}$;
    \item as arrows, the morphisms of $\cat{A}$;
    \item as proarrows $a \proTo b$, \define{$L$-structured cospans} with feet
      $a$ and $b$, which are cospans in $\cat{X}$ of the form
      $La \rightarrow x \leftarrow Lb$;
    \item as cells $\inlineCell{a}{b}{c}{d}{x}{y}{f}{g}{}$, \define{morphisms of
      $L$-structured cospans} with foot maps $f$ and $g$, which are morphisms of
      cospans in $\cat{X}$ of the form
      \begin{equation*}
        \begin{tikzcd}
          {L(a)} & x & {L(b)} \\
          {L(d)} & y & {L(c)}
          \arrow["Lf"', from=1-1, to=2-1]
          \arrow[from=1-1, to=1-2]
          \arrow[from=1-3, to=1-2]
          \arrow["h"', from=1-2, to=2-2]
          \arrow[from=2-1, to=2-2]
          \arrow["Lg", from=1-3, to=2-3]
          \arrow[from=2-3, to=2-2]
        \end{tikzcd}.
      \end{equation*}
  \end{itemize}
  Composition in the categories $\SCsp{L}{\cat{X}}_0$ and $\SCsp{L}{\cat{X}}_1$
  is by composition in $\cat{A}$ and $\Csp{\cat{X}}_1$, respectively, and
  external composition in $\SCsp{L}{\cat{X}}$ is given by that in
  $\Csp{\cat{X}}$, i.e., by pushout in $\cat{X}$.
\end{proposition}

We take this result as given. The proof is straightforward because the double
category structure of $L$-structured cospans is inherited from that of cospans
in $\cat{X}$. For details, see \cite[Theorem 2.3]{baez2020}.

\subsection{Cocartesian Equipment of Structured Cospans}

We now prove that structured cospans form an equipment, then a cocartesian
double category, and hence a cocartesian equipment. The reader may find it
helpful to review the definitions of cocartesian double categories and
equipments in \cref{app:cocartesian-equipments}.

\begin{proposition}
  Let $L: \cat{A} \to \cat{X}$ be a functor into a category $\cat{X}$ with
  pushouts. Then the double category of $L$-structured cospans is an equipment.
\end{proposition}
\begin{proof}
  To restrict an $L$-structured cospan $(c, Lc \rightarrow y \leftarrow Ld, d)$
  along arrows ${f: a \to c}$ and ${g: b \to d}$ in $\cat{A}$, simply restrict
  the underlying cospan in $\cat{X}$ along $Lf$ and $Lg$, using the fact that
  $\Csp{\cat{X}}$ is an equipment (\cref{ex:cospan-dbl-cat}). The universal
  property holds as a special case of the universal property in $\Csp{\cat{X}}$:
  \begin{equation*}
    \begin{tikzcd}
      {La'} & x & {Lb'} \\
      La && Lb \\
      Lc & y & Ld
      \arrow["\ell"', from=3-1, to=3-2]
      \arrow["r", from=3-3, to=3-2]
      \arrow["{L g'}", from=1-3, to=2-3]
      \arrow["{L f'}"', from=1-1, to=2-1]
      \arrow[from=1-1, to=1-2]
      \arrow[from=1-3, to=1-2]
      \arrow["h"', from=1-2, to=3-2]
      \arrow["Lf"', from=2-1, to=3-1]
      \arrow["Lg", from=2-3, to=3-3]
    \end{tikzcd}
    \qquad=\qquad
    \begin{tikzcd}
      {L a'} & x & {L b'} \\
      La & y & Lb \\
      Lc & y & Ld
      \arrow["Lf"', from=2-1, to=3-1]
      \arrow["{L f'}"', from=1-1, to=2-1]
      \arrow["{L g'}", from=1-3, to=2-3]
      \arrow["Lg", from=2-3, to=3-3]
      \arrow["\ell"', from=3-1, to=3-2]
      \arrow["r", from=3-3, to=3-2]
      \arrow[Rightarrow, no head, from=2-2, to=3-2]
      \arrow["{\ell \circ Lf}"', from=2-1, to=2-2]
      \arrow["{r \circ Lg}", from=2-3, to=2-2]
      \arrow["h"', dashed, from=1-2, to=2-2]
      \arrow[from=1-1, to=1-2]
      \arrow[from=1-3, to=1-2]
    \end{tikzcd}
    \qedhere
  \end{equation*}
\end{proof}

For the double category of $L$-structured cospans to be cocartesian, extra
assumptions are needed. Clearly, the category $\cat{A}$ must itself have finite
coproducts. Also, these must preserved by the functor $L: \cat{A} \to \cat{X}$.
The latter is often is easy to verify in examples by exhibiting $L$ as a left
adjoint.

\begin{theorem} \label{thm:structured-csp-cocartesian}
  Suppose $\cat{A}$ is a category with finite coproducts, $\cat{X}$ is a
  category with finite colimits, and $L: \cat{A} \to \cat{X}$ is a functor that
  preserves finite coproducts. Then the double category of $L$-structured
  cospans is cocartesian, hence also a cocartesian equipment.
\end{theorem}
\begin{proof}
  Because the categories $\cat{A}$ and $\cat{X}$ have finite coproducts, there
  are canonical comparison maps
  \begin{equation*}
    L_{a,a'} \coloneqq [L(\iota_a), L(\iota_{a'})]: L(a) + L(a') \to L(a+a'),
    \qquad a, a' \in \cat{A},
  \end{equation*}
  and $L_0 \coloneqq {!_{L(0)}}: 0_{\cat{X}} \to L(0_{\cat{A}})$. For any maps
  $f:a \to c$ and $f': a' \to c$ in $\cat{A}$, the comparisons satisfy
  \begin{equation*}
    \begin{tikzcd}[row sep=small, column sep=small]
      {L(a)+L(a')} && {L(a+a')} \\
      & {L(c)}
      \arrow["{[Lf, Lf']}"', from=1-1, to=2-2]
      \arrow["{L([f,f'])}", from=1-3, to=2-2]
      \arrow["{L_{a,a'}}", from=1-1, to=1-3]
    \end{tikzcd}
  \end{equation*}
  as shown by precomposing both sides with the coprojections $\iota_{La}$ and
  $\iota_{La'}$ to obtain $Lf$ and $Lf'$, respectively. Since by assumption $L$
  preserves finite coproducts, the comparisons $L_{a,a'}$ and $L_0$ are, in
  fact, isomorphisms.

  We now prove that the categories underlying $\SCsp{L}{\cat{X}}$ are
  cocartesian. By assumption, the category $\SCsp{L}{\cat{X}}_0 = \cat{A}$ has
  finite coproducts. Since the comparison $L_0$ is an isomorphism,
  $L(0_{\cat{A}})$ is initial in $\cat{X}$ and the initial $L$-structured cospan
  is $(0_{\cat{A}}, \id_{L(0_{\cat{A}})}, 0_{\cat{A}})$. Furthermore, the
  coproduct of two $L$-structured cospans
  ${(a, La \rightarrow x \leftarrow Lb, b)}$ and
  ${(a', La' \rightarrow x' \leftarrow Lb', b')}$, denoted
  \begin{equation*}
    (a+a',\ L(a+a') \rightarrow x+x' \leftarrow L(b+b'),\ b+b'),
  \end{equation*}
  is obtained from the pointwise coproduct of cospans in $\cat{X}$ by
  restriction along the inverse comparisons $L_{a,a'}^{-1}$ and $L_{b,b'}^{-1}$.
  The universal property of coproducts in $\SCsp{L}{\cat{X}}_1$ then takes the
  form:
  \begin{equation*}
    \begin{tikzcd}
      {L(a+a')} & {L(a) +L(a')} & {x+x'} & {L(b)+L(b')} & {L(b+b')} \\
      {L(c)} & {L(c)} & y & {L(d)} & {L(d)}
      \arrow[from=1-2, to=1-3]
      \arrow["{L_{a,a'}^{-1}}", from=1-1, to=1-2]
      \arrow[from=1-4, to=1-3]
      \arrow["{L_{b,b'}^{-1}}"', from=1-5, to=1-4]
      \arrow["{L([f,f'])}"', from=1-1, to=2-1]
      \arrow[Rightarrow, no head, from=2-1, to=2-2]
      \arrow["{[h,h']}"', from=1-3, to=2-3]
      \arrow["{[Lf,Lf']}"', from=1-2, to=2-2]
      \arrow[from=2-2, to=2-3]
      \arrow[from=2-4, to=2-3]
      \arrow["{[Lg,Lg']}", from=1-4, to=2-4]
      \arrow[Rightarrow, no head, from=2-5, to=2-4]
      \arrow["{L([g,g'])}", from=1-5, to=2-5]
    \end{tikzcd}.
  \end{equation*}
  We have shown that both categories underlying $\SCsp{L}{\cat{X}}$ have finite
  coproducts, and it is immediate that the source and target functors preserve
  them.

  Finally, the comparison cells in $\SCsp{L}{\cat{X}}$ interchanging finite
  coproducts with external composition and identity
  (\cref{def:cocartesian-dbl-cat}) are all isomorphisms because they are defined
  by the same maps in $\cat{X}$ as the comparison cells in $\Csp{\cat{X}}$,
  which we already know to be isomorphisms (\cref{ex:cospan-dbl-cat}).
\end{proof}

As a corollary, every double category of $L$-structured cospans satisfying the
hypotheses of the theorem can be given the structure of a symmetric monoidal
double category, by making choices of coproducts in both underlying categories.
This follows abstractly because any cocartesian object in a 2-category with
finite 2-products is a symmetric pseudomonoid in a canonical way \cite[Remark
2.11]{shulman2010}. Cocartesian double categories are cocartesian objects in the
2-category $\Dbl$, whereas symmetric monoidal double categories are symmetric
pseudomonoids in $\Dbl$.

\subsection{Maps Between Structured Cospan Double Categories}

We complete the essential theory of structured cospans by showing how to
construct maps between cocartesian equipments of structured cospans. These maps
are cocartesian double functors (\cref{def:cocartesian-dbl-functor}). Compared
with the original results \cite[Theorems 4.2 and 4.3]{baez2020}, the theorem
below is slightly more general, treating the lax case as well as the pseudo one,
and slightly stronger, yielding cocartesian double functors instead of symmetric
monoidal ones.

\begin{theorem} \label{thm:maps-structured-csp}
  Suppose we have a diagram in $\Cat$ of the form
  \begin{equation*}
    \begin{tikzcd}
      {\cat{A}} & {\cat{X}} \\
      {\cat{A}'} & {\cat{X}'}
      \arrow["L", from=1-1, to=1-2]
      \arrow["{L'}"', from=2-1, to=2-2]
      \arrow["{F_1}", from=1-2, to=2-2]
      \arrow["{F_0}"', from=1-1, to=2-1]
      \arrow["\alpha", shorten <=4pt, shorten >=4pt, Rightarrow, from=2-1, to=1-2]
    \end{tikzcd},
  \end{equation*}
  where the categories $\cat{X}$ and $\cat{X}'$ have pushouts. Then there is a
  lax double functor $\dbl{F}: \SCsp{L}{\cat{X}} \to \SCsp{L'}{\cat{X}'}$ that
  has underlying functor $\dbl{F}_0 = F_0$ and acts on proarrows as
  \begin{equation*}
    (a, L(a) \xrightarrow{\ell} x \xleftarrow{r} L(b), b)
    \quad\mapsto\quad
    (F_0(a), L'(F_0(a)) \xrightarrow{\alpha_a} F_1(L(a)) \xrightarrow{F_1(\ell)}
    F_1(x) \xleftarrow{F_1(r)} F_1(L(b)) \xleftarrow{\alpha_b} L'(F_0(b)), F_0(b))
  \end{equation*}
  and on cells as
  \begin{equation*}
    \begin{tikzcd}
      {L(a)} & x & {L(b)} \\
      {L(a')} & {x'} & {L(b')}
      \arrow["\ell", from=1-1, to=1-2]
      \arrow["r"', from=1-3, to=1-2]
      \arrow["Lf"', from=1-1, to=2-1]
      \arrow["{\ell'}"', from=2-1, to=2-2]
      \arrow["{r'}", from=2-3, to=2-2]
      \arrow["Lg", from=1-3, to=2-3]
      \arrow["h"', from=1-2, to=2-2]
    \end{tikzcd}
    \quad\mapsto\quad
    \begin{tikzcd}
      {L'(F_0(a))} & {F_1(x)} & {L'(F_0(b))} \\
      {L'(F_0(a'))} & {F_1(x')} & {L'(F_0(b'))}
      \arrow["{F_1(\ell)\circ \alpha_a}", from=1-1, to=1-2]
      \arrow["{F_1(r) \circ \alpha_b}"', from=1-3, to=1-2]
      \arrow["{L'(F_0(f))}"', from=1-1, to=2-1]
      \arrow["{F_1(\ell') \circ \alpha_{a'}}"', from=2-1, to=2-2]
      \arrow["{F_1(r') \circ \alpha_{b'}}", from=2-3, to=2-2]
      \arrow["{L'(F_0(g))}", from=1-3, to=2-3]
      \arrow["{F_1(h)}"', from=1-2, to=2-2]
    \end{tikzcd}.
  \end{equation*}
  Moreover, $\dbl{F}$ is a pseudo double functor whenever $F_1$ preserves
  pushouts and $\alpha$ is a natural isomorphism.

  Suppose further that all of the categories in question have finite coproducts
  and that $L$ and $L'$ preserve them, so that both double categories
  $\SCsp{L}{\cat{X}}$ and $\SCsp{L'}{\cat{X}'}$ are cocartesian. Then the lax
  double functor $\dbl{F}$ is cocartesian if and only if both functors $F_0$ and
  $F_1$ are cocartesian. In particular, $\dbl{F}$ is a cocartesian pseudo double
  functor whenever $F_0$ preserves finite coproducts, $F_1$ preserves finite
  colimits, and $\alpha$ is a natural isomorphism.
\end{theorem}

As a substantial application of the theorem, we have formulated the generalized
Lokta-Volterra model as a cocartesian lax double functor from open signed graphs
to open parameterized dynamical systems \cite{aduddell2023}.

We prove the theorem by decomposing the lax double functor $\dbl{F}$ into three
simpler ones. Taken together, the lemmas also implicitly give formulas for the
laxators and unitors of $\dbl{F}$, which we omitted in the theorem statement.

\begin{lemma}
  Let $\cat{X}$ be a category with pushouts and let
  $\cat{A}_0 \xrightarrow{F_0} \cat{A} \xrightarrow{L} \cat{X}$ be functors.
  Then there is a \emph{strict} double functor
  $\SCsp{F_0}{\cat{X}}: \SCsp{L \circ F_0}{\cat{X}} \to \SCsp{L}{\cat{X}}$ given
  by $F_0$ on objects and arrows and by the identity on the cospans and maps of
  cospans underlying proarrows and cells.

  Furthermore, the double functor $\SCsp{F_0}{\cat{X}}$ is cocartesian whenever
  $\cat{A}_0$, $\cat{A}$, and $\cat{X}$ have finite coproducts and the functors
  $L$ and $F_0$ preserve them.
\end{lemma}

The proof is immediate from the definitions. The next lemma is slightly more
involved.

\begin{lemma}
  Let $\cat{X}$ and $\cat{X}'$ be categories with pushouts and let
  $\cat{A} \xrightarrow{L} \cat{X} \xrightarrow{F_1} \cat{X}'$ be functors. Then
  there is a \emph{normal} lax double functor
  $\SCsp{L}{F_1}: \SCsp{L}{\cat{X}} \to \SCsp{F_1 \circ L}{\cat{X}'}$ that is
  the identity on objects and arrows and acts on proarrows and cells by
  postcomposing the underlying diagrams in $\cat{X}$ with
  $F_1: \cat{X} \to \cat{X}'$.

  The laxators are given by the universal property of pushouts in $\cat{X}'$,
  and $\SCsp{L}{F_1}$ is pseudo if and only if $F_1$ preserves pushouts.
  Furthermore, when $\cat{A}$, $\cat{X}$, and $\cat{X}'$ have finite coproducts
  and $L$ preserves them, $\SCsp{L}{F_1}$ is cocartesian if and only if $F_1$ is
  cocartesian.
\end{lemma}
\begin{proof}
  The proposed lax double functor
  $\SCsp{L}{F_1}: \SCsp{L}{\cat{X}} \to \SCsp{F_1 \circ L}{\cat{X}'}$ acts on
  cospans and maps of cospans in exactly the same way as the lax double functor
  $\Csp{F_1}: \Csp{\cat{X}} \to \Csp{\cat{X}'}$ reviewed in
  \cref{ex:maps-cospan-dbl-cat}. The proof thus carries over directly.
\end{proof}

In the final lemma, we isolate the maps between structured cospan double
categories induced by natural transformations between the structuring functors.

\begin{lemma}
  Let $\cat{X}$ be a category with pushouts and let
  $\alpha: L' \Rightarrow L: \cat{A} \to \cat{X}$ be a natural transformation.
  Then there is a lax double functor
  $\alpha^*: \SCsp{L}{\cat{X}} \to \SCsp{L'}{\cat{X}}$ that acts
  \begin{itemize}[noitemsep]
    \item on objects and arrows, as the identity;
    \item on proarrows $a \proTo b$, by restricting the underlying cospan
      $L(a) \rightarrow x \leftarrow L(b)$ along the components
      $\alpha_a: L'(a) \to L(a)$ and $\alpha_b: L'(b) \to L(b)$;
    \item on cells $\inlineCell{a}{b}{c}{d}{m}{n}{f}{g}{}$, by pasting the
      naturality squares for $f$ and $g$:
      \useshortskip
      \begin{equation*}
        \begin{tikzcd}
          {L'(a)} & {L(a)} & x & {L(b)} & {L'(b)} \\
          {L'(c)} & {L(c)} & y & {L(d)} & {L'(d)}
          \arrow["{L(f)}"', from=1-2, to=2-2]
          \arrow["{\alpha_a}", from=1-1, to=1-2]
          \arrow["{L'(f)}"', from=1-1, to=2-1]
          \arrow["{\alpha_c}"', from=2-1, to=2-2]
          \arrow[from=1-2, to=1-3]
          \arrow[from=2-2, to=2-3]
          \arrow[from=2-4, to=2-3]
          \arrow[from=1-4, to=1-3]
          \arrow["{\alpha_b}"', from=1-5, to=1-4]
          \arrow["{L(g)}", from=1-4, to=2-4]
          \arrow["h"', from=1-3, to=2-3]
          \arrow["{L'(g)}", from=1-5, to=2-5]
          \arrow["{\alpha_d}", from=2-5, to=2-4]
        \end{tikzcd}.
      \end{equation*}
  \end{itemize}
  The laxator
  $\alpha^*_{m,n}: \alpha^*(m) \odot \alpha^*(n) \to \alpha^*(m \odot n)$ for
  proarrows $m = (a,\; L(a) \rightarrow x \leftarrow L(b),\; b)$ and
  $n = (b,\; L(b) \rightarrow y \leftarrow L(c),\; c)$ has apex map given by the
  universal property of the pushout over $L'(b)$:
  \begin{equation*}
    \begin{tikzcd}
      && y \\
      {L'(b)} & {L(b)} & {x +_{L'(b)} y} & {x +_{L(b)} y} \\
      && x
      \arrow["{\iota_x}"', from=3-3, to=2-4]
      \arrow["{\iota'_x}", from=3-3, to=2-3]
      \arrow["{\iota'_y}"', from=1-3, to=2-3]
      \arrow["{\iota_y}", from=1-3, to=2-4]
      \arrow[dashed, from=2-3, to=2-4]
      \arrow["{\alpha_b}", from=2-1, to=2-2]
      \arrow[from=2-2, to=3-3]
      \arrow[from=2-2, to=1-3]
    \end{tikzcd}.
  \end{equation*}
  The unitor $\alpha^*_a: \id_a' \to \alpha^*(\id_a)$ for object $a \in \cat{A}$
  has apex map $\alpha_a: L'(a) \to L(a)$. The lax double functor $\alpha^*$ is
  pseudo whenever $\alpha$ is a natural isomorphism, and it is automatically
  cocartesian whenever the structured cospan double categories are cocartesian.
\end{lemma}
\begin{proof}
  The laxators and unitors obey the coherence axioms by the uniqueness part of
  the universal property. Importantly, the last statement about cocartesianness
  holds because natural transformations automatically commute with coproducts.
  That is, if $\cat{A}$ and $\cat{X}$ have finite coproducts, then, using the
  notation of the proof of \cref{thm:structured-csp-cocartesian}, the following
  diagrams commute for all objects $a, b \in \cat{A}$:
  \begin{equation*}
    \begin{tikzcd}
      {L'(a)+L'(b)} & {L'(a+b)} \\
      {L(a)+L(b)} & {L(a+b)}
      \arrow["{L'_{a,b}}", from=1-1, to=1-2]
      \arrow["{\alpha_a + \alpha_b}"', from=1-1, to=2-1]
      \arrow["{\alpha_{a+b}}", from=1-2, to=2-2]
      \arrow["{L_{a,b}}"', from=2-1, to=2-2]
    \end{tikzcd}
    \qquad\text{and}\qquad
    \begin{tikzcd}[row sep=tiny]
      & {L'(0_{\cat{A}})} \\
      {0_{\cat{X}}} \\
      & {L(0_{\cat{A}})}
      \arrow["{L'_0}", from=2-1, to=1-2]
      \arrow["{L_0}"', from=2-1, to=3-2]
      \arrow["{\alpha_0}", from=1-2, to=3-2]
    \end{tikzcd}.
  \end{equation*}
  Restricting along the components of $\alpha$ thus commutes with restricting
  along the inverse comparison maps and so also commutes with coproducts of
  structured cospans.
\end{proof}

\begin{proof}[Proof of \cref{thm:maps-structured-csp}]
  Using the three lemmas, the lax double functor
  $\dbl{F}: \SCsp{L}{\cat{X}} \to \SCsp{L'}{\cat{X}'}$ is realized as the
  composite
  \begin{equation*}
    \dbl{F}: \SCsp{L}{\cat{X}}
      \xrightarrow{\SCsp{L}{F_1}} \SCsp{F_1 \circ L}{\cat{X}'}
      \xrightarrow{\alpha^*} \SCsp{L' \circ F_0}{\cat{X}'}
      \xrightarrow{\SCsp{F_0}{\cat{X'}}} \SCsp{L'}{\cat{X}'}.
    \qedhere
  \end{equation*}
\end{proof}

\section{Decorated Cospans as a Double Grothendieck Construction}
\label{sec:decorated-cospans}

Decorated cospans represent open systems as cospans with apexes decorated by
extra data \cite{fong2015,baez2022}. For example, open dynamical systems
comprise a cospan of finite sets along with a dynamical system whose set of
state variables is the apex set \cite{baez2017}. In contrast to structured
cospans, the symmetric monoidal product of decorated cospans need not satisfy a
universal property such as cocartesianness. Decorated cospans are therefore
applicable in certain situations where structured cospans are not, at the
expense of requiring more data to construct.

The Grothendieck construction $\int F$ of a functor $F: \cat{A} \to \Cat$ can be
thought to decorate the objects of $\cat{A}$ with data from $F$, inasmuch as the
objects of $\int F$ consist of an object $a \in \cat{A}$ together with an object
$x \in F(a)$ (the ``decoration''). So one might suppose that decorated cospans
arise from a Grothendieck construction. For that to be the case, the cospans
being decorated must be the \emph{objects} of a category. Fortunately, as we
emphasized in \cref{sec:introduction}, that is precisely how cospans are seen by
the modern theory of double categories. In this section, we reconstruct and
generalize the double category of decorated cospans using the double-categorical
analogue of the Grothendieck construction.

\subsection{Double Grothendieck construction}

In their study of double fibrations \cite{cruttwell2022}, Cruttwell, Lambert,
Pronk, and Szyld introduced a Grothendieck construction for double categories,
taking as input a lax double functor into $\Span{\Cat}$.\footnote{In its most
  general form, the double Grothendieck construction takes as input a lax double
  \emph{pseudo} functor into $\Span{\Cat}$, analogous to how the Grothendieck
  construction takes a \emph{pseudo}functor into $\Cat$. For simplicity, we
  eschew this aspect but see \cite[Definition 3.12]{cruttwell2022}.}

Before stating the construction, we unpack some of the considerable amount of
data contained in a lax double functor $F: \dbl{A} \to \Span{\Cat}$. First,
there are natural transformations
\begin{equation*}
  \sigma: \apex \circ F_1 \Rightarrow F_0 \circ \src: \dbl{A}_1 \to \Cat
  \qquad\text{and}\qquad
  \tau: \apex \circ F_1 \Rightarrow F_0 \circ \tgt: \dbl{A}_1 \to \Cat
\end{equation*}
whose components are the functors $\sigma_m$ and $\tau_m$ defined by
\begin{equation*}
  F_1(m) \eqqcolon \left(
    F_0(a) = \leftfoot(F_1(m)) \xleftarrow{\sigma_m}
    \apex(F_1(m)) \xrightarrow{\tau_m} \rightfoot(F_1(m)) = F_0(b)
  \right)
\end{equation*}
for each proarrow $m: a \proTo b$ in $\dbl{A}$. The naturality squares for
$\sigma$ and $\tau$ are precisely the maps of spans
\begin{equation*}
  F_1(\alpha) = \left(
    \begin{tikzcd}
      {F_0(a)} & {\leftfoot(F_1(m))} & {\apex(F_1(m))} & {\rightfoot(F_1(m))} & {F_0(b)} \\
      {F_0(c)} & {\leftfoot(F_1(n))} & {\apex(F_1(n))} & {\rightfoot(F_1(n))} & {F_0(d)}
      \arrow["{\apex(F_1(\alpha))}"', from=1-3, to=2-3]
      \arrow["{\sigma_m}"', from=1-3, to=1-2]
      \arrow["{F_0(f)}"', from=1-1, to=2-1]
      \arrow["{\sigma_n}", from=2-3, to=2-2]
      \arrow[Rightarrow, no head, from=1-2, to=1-1]
      \arrow[Rightarrow, no head, from=2-2, to=2-1]
      \arrow["{\leftfoot(F_1(\alpha))}"', from=1-2, to=2-2]
      \arrow["{\tau_m}", from=1-3, to=1-4]
      \arrow["{\tau_n}"', from=2-3, to=2-4]
      \arrow["{\rightfoot(F_1(\alpha))}", from=1-4, to=2-4]
      \arrow[Rightarrow, no head, from=1-4, to=1-5]
      \arrow["{F_0(g)}", from=1-5, to=2-5]
      \arrow[Rightarrow, no head, from=2-4, to=2-5]
    \end{tikzcd}
  \right)
\end{equation*}
for each cell $\inlineCell{a}{b}{c}{d}{m}{n}{f}{g}{\alpha}$ in $\dbl{A}$.
Writing $F_{m,n}: F(m) \odot F(n) \to F(m \odot n)$ and
$F_a: \id_{Fa} \to F(\id_a)$ for the laxators and unitors of $F$, there are also
natural families of functors
\begin{equation*}
  \Phi_{m,n} \coloneqq \apex(F_{m,n}):
    \apex(F(m)) \prescript{}{\tau_m}{\times_{\sigma_n}} \apex(F(n)) \to
    \apex(F(m \odot n))
\end{equation*}
and $\Phi_a \coloneqq \apex(F_a): F(a) \to \apex(F(\id_x))$, indexed by
proarrows $a \xproTo{m} b \xproTo{n} c$ and objects $a$ in $\dbl{A}$.

Using this notation, the double Grothendieck construction \cite[Theorem
3.51]{cruttwell2022} appears as:

\begin{theorem} \label{thm:dbl-grothendieck}
  Given a lax double functor $F: \dbl{A} \to \Span{\Cat}$, there is a double
  category $\int F$, the \define{double Grothendieck construction} of $F$, with
  underlying categories ${(\int F)_0 = \int F_0}$ and
  ${(\int F)_1 = \int (\apex \circ F_1)}$. Explicitly, the double category
  $\int F$ has
  \begin{itemize}
    \item as objects, pairs $(a,x)$ where $a$ is an object of $\dbl{A}$ and $x$
      is an object of $F(a)$;
    \item as arrows $(a,x) \to (b,y)$, pairs $(f, \phi)$ where $f: a \to b$ is
      an arrow of $\dbl{A}$ and $\phi: F(f)(x) \to y$ is a morphism of $F(b)$;
    \item as proarrows $(a,x) \proTo (b,y)$, pairs $(m, s)$ where
      $m: a \proTo b$ is a proarrow of $\dbl{A}$ and $s$ is an object of
      $\apex(F(m))$ such that $\sigma_m(s) = x$ and $\tau_m(s) = y$;
    \item as cells
      $\inlineCell{(a,x)}{(b,y)}{(c,w)}{(d,z)}{(m,s)}{(n,t)}{(f,\phi)}{(g,\psi)}{}$,
      pairs $(\alpha, \nu)$ such that
      $\inlineCell{a}{b}{c}{d}{m}{n}{f}{g}{\alpha}$ is a cell in $\dbl{A}$ and
      ${\nu: \apex(F(\alpha))(s) \to t}$ is a morphism of $\apex(F(n))$ such
      that $\sigma_n(\nu) = \phi$ and $\tau_n(\nu) = \psi$.
  \end{itemize}
  External composition and identities in $\int F$ are as follows.
  \begin{itemize}
    \item The composite of proarrows
      $(a,x) \xproTo{(m,s)} (b,y) \xproTo{(n,t)} (c,z)$ is
      $(m \odot n, \Phi_{m,n}(s,t)): (a,x) \proTo (b,y)$.
    \item The external composite of cells is
    \begin{equation*}
      \begin{tikzcd}
        {(a,x)} & {(b,y)} & {(c,z)} \\
        {(a',x')} & {(b',y')} & {(c',z')}
        \arrow[""{name=0, anchor=center, inner sep=0}, "{(m,s)}", "\shortmid"{marking}, from=1-1, to=1-2]
        \arrow[""{name=1, anchor=center, inner sep=0}, "{(n,t)}", "\shortmid"{marking}, from=1-2, to=1-3]
        \arrow[""{name=2, anchor=center, inner sep=0}, "{(m', s')}"', "\shortmid"{marking}, from=2-1, to=2-2]
        \arrow["{(f,\phi)}"', from=1-1, to=2-1]
        \arrow[""{name=3, anchor=center, inner sep=0}, "{(n', t')}"', "\shortmid"{marking}, from=2-2, to=2-3]
        \arrow["{(g,\psi)}"{description}, from=1-2, to=2-2]
        \arrow["{(h,\eta)}", from=1-3, to=2-3]
        \arrow["{(\alpha, \mu)}"{description}, draw=none, from=0, to=2]
        \arrow["{(\beta,\nu)}"{description}, draw=none, from=1, to=3]
      \end{tikzcd}
      \qquad\coloneqq\qquad
      \begin{tikzcd}
        {(a,x)} && {(c,z)} \\
        {(a',x')} && {(c',z')}
        \arrow["{(f,\phi)}"', from=1-1, to=2-1]
        \arrow["{(h,\eta)}", from=1-3, to=2-3]
        \arrow[""{name=0, anchor=center, inner sep=0}, "{(m \odot n, \Phi_{m,n}(s,t))}", "\shortmid"{marking}, from=1-1, to=1-3]
        \arrow[""{name=1, anchor=center, inner sep=0}, "{(m' \odot n', \Phi_{m',n'}(s',t'))}"', "\shortmid"{marking}, from=2-1, to=2-3]
        \arrow["{(\alpha \odot \beta, \Phi_{m',n'}(\mu,\nu))}"{description}, draw=none, from=0, to=1]
      \end{tikzcd}.
    \end{equation*}
    \item The identity proarrow at object $(a,x)$ is $(\id_a, \Phi_a(x))$.
    \item The identity cell at arrow $(f,\phi): (a,x) \to (b,y)$ is
      $(\id_f, \Phi_b(\phi))$.
  \end{itemize}
  Moreover, there is a canonical \define{projection} $\pi_F: \int F \to \dbl{A}$,
  which is a strict double functor.
\end{theorem}

\subsection{A Modular Reconstruction of Decorated Cospans}

To define decorated cospans, we apply the double Grothendieck construction in
the case that the base double category $\dbl{A}$ is a double category of
cospans. Specifically, let $\cat{A}$ be a category with pushouts and let
$F: \Csp{\cat{A}} \to \Span{\Cat}$ be a lax double functor. Then the
\define{double category of $F$-decorated cospans}, denoted $\DCsp{F}$, is the
double Grothendieck construction $\int F$.

This notion of decorated cospan is more general than the established one
\cite[\S 2]{baez2022} in two different ways. First, the decorations assigned to
a cospan may depend on the whole cospan, not just on its apex. Second, the feet
of the cospans receive their own decorations, which can be extracted from the
cospan decorations using the transformations denoted $\sigma$ and $\tau$ above.
For two decorated cospans to be composable, not only must the feet of the
cospans be compatible, so must be the decorations on the feet. We will see an
application that takes advantage of this extra generality shortly. Before that,
we show how to recover the original notion of decorated cospan based on lax
monoidal functors into $(\Cat,\times)$.

\begin{corollary}
  Let $\cat{A}$ be a category with finite colimits and let
  $F: (\cat{A}, +) \to (\Cat, \times)$ be a lax monoidal functor. Then there is
  a double category $\DCsp{F}$ that has
  \begin{itemize}[noitemsep]
    \item as objects, the objects of $\cat{A}$;
    \item as arrows, the morphisms of $\cat{A}$;
    \item as proarrows $a \proTo b$, \define{$F$-decorated cospans} with feet
      $a$ and $b$, which are cospans $p = (a \rightarrow m \leftarrow b)$ in
      $\cat{A}$ together with a \define{decoration} $s \in F(m)$;
    \item as cells $\inlineCell{a}{b}{c}{d}{(p,s)}{(q,t)}{f}{g}{}$ where
      $p = (a \rightarrow m \leftarrow b)$ and
      $q = (c \rightarrow n \leftarrow d)$, \define{morphisms of $F$-decorated
      cospans} with foot maps $f$ and $g$, which are morphisms of cospans in
      $\cat{A}$ of the form
      \begin{equation*}
        \begin{tikzcd}
          a & m & c \\
          b & n & d
          \arrow[from=1-1, to=1-2]
          \arrow[from=1-3, to=1-2]
          \arrow[from=2-1, to=2-2]
          \arrow[from=2-3, to=2-2]
          \arrow["f"', from=1-1, to=2-1]
          \arrow["g", from=1-3, to=2-3]
          \arrow["h"', from=1-2, to=2-2]
        \end{tikzcd}
      \end{equation*}
      together with a \define{decoration morphism} $\nu: F(h)(s) \to t$ in
      $F(n)$.
  \end{itemize}
  The composite of proarrows $a \xproTo{(p,s)} b \xproTo{(q,t)} c$, where
  $p = (a \rightarrow m \leftarrow b)$ and $q = (b \rightarrow n \leftarrow c)$,
  is the proarrow $(p \odot q, \Phi_{m,n}(s,t))$, where the cospan $p \odot q$
  is given by pushout in $\cat{A}$ and the functor $\Phi_{m,n}$ is the composite
  \begin{equation*}
    \Phi_{m,n}: F(m) \times F(n) \xrightarrow{F_{m,n}}
      F(m+n) \xrightarrow{F([\iota_m, \iota_n])}
      F(m +_b n).
  \end{equation*}
  The identity proarrow at $a \in \cat{A}$ is $(\id_a, \Phi_a)$, where $\Phi_a$
  is the composite $\cat{1} \xrightarrow{F_0} F(0) \xrightarrow{F(!)} F(a)$.

  Moreover, there is a canonical \define{projection}
  $\pi_F: \DCsp{F} \to \Csp{\cat{A}}$, which is a strict double functor.
\end{corollary}
\begin{proof}
  We construct the double category $\DCsp{F}$ in a modular fashion by applying
  the double Grothendieck construction to a lax double functor
  $\tilde F: \Csp{\cat{A}} \to \Span{\Cat}$ that is itself the composite of
  three simpler lax double functors:
  \begin{equation*}
    \tilde F: \Csp{\cat{A}} \xrightarrow{\Apex}
      \MonDbl{(\cat{A}, +)} \xrightarrow{\MonDbl{F}}
      \MonDbl{(\Cat, \times)} \xrightarrow{\Apex_*}
      \Span{\Cat}.
  \end{equation*}
  Let us explain each of these. First, any monoidal category
  $(\cat{C}, \otimes, I)$ can be regarded as a double category $\dbl{D}$ whose
  category of objects is trivial, $\dbl{D}_0 = \cat{1}$; whose category of
  morphisms is $\dbl{D}_1 = \cat{C}$; and whose external composition and
  identity are the monoidal product and unit \cite[\S 3.3.4]{grandis2019}. Lax
  monoidal functors then induce lax double functors between such degenerate
  double categories, and monoidal natural transformations induce natural
  transformations of those, so altogether there is a 2-functor
  ${\MonDbl{}: \MonCatLax \to \DblLax}$. In particular, the lax monoidal functor
  $F: (\cat{A},+) \to (\Cat,\times)$ induces a lax double functor $\MonDbl{F}$.

  Next, given a category $\cat{A}$ with finite colimits, the lax double functor
  $\Apex: \Csp{\cat{A}} \to \MonDbl{(\cat{A},+)}$ has the unique map
  $\Apex_0: \cat{A} \xrightarrow{!} \cat{1}$ between categories of objects and
  the functor
  $\Apex_1 := {\cat{A}^{\{\bullet \rightarrow \bullet \leftarrow \bullet\}} \xrightarrow{\apex} \cat{A}}$
  between categories of morphisms. The laxators
  \begin{equation*}
    \Apex_{p,q}: \Apex(p) + \Apex(q) = m + n
      \xrightarrow{[\iota_m,\iota_n]}
      m +_{b} n = \Apex(p \odot n)
  \end{equation*}
  for proarrows $p = (a \rightarrow m \leftarrow b)$ and
  $q = (b \rightarrow n \leftarrow c)$, and the unitors
  $\Apex_a: 0 \xrightarrow{!} a$ for objects $a \in \cat{A}$, are all given by
  the universal properties of the colimits involved.

  Finally, given a category $\cat{C}$ with finite limits, the double functor
  $\Apex_*: \MonDbl{(\cat{C},\times)} \to \Span{\cat{C}}$ has underlying
  functors $(\Apex_*)_0: \cat{1} \to \cat{C}$ picking out the terminal object
  $1$ of $\cat{C}$ and
  $(\Apex_*)_1: \cat{C} \to \cat{C}^{\{\bullet \leftarrow \bullet \rightarrow \bullet\}}$
  sending each object $c \in \cat{C}$ to the span
  $1 \xleftarrow{!} c \xrightarrow{!} 1$. This double functor is pseudo because
  products are isomorphic to pullbacks over the terminal object. By making
  reasonable choices of products and pullbacks, we can even assume that the
  double functor is strict.

  The double category $\DCsp{F}$ is precisely the double Grothendieck
  construction of $\tilde F$ (\cref{thm:dbl-grothendieck}). This follows from
  the formulas for the laxators and unitors of a composite lax double functor
  \cite[Equation 3.63]{grandis2019}. In terms of the notation in the corollary
  statement, the laxators and unitors of the composite $\tilde F$ are
  $\tilde F_{p,q} = \Apex_*(\Phi_{\Apex(p),\Apex(q)})$ and
  $\tilde F_a = \Apex_*(\Phi_a)$.
\end{proof}

This result was first proved in \cite[Theorem 2.1]{baez2022}. Our reconstruction
solves a lingering conceptual puzzle about the composition law for decorated
cospans: why does it involve two operations, instead of just one? As the proof
shows, the reason is that decorated cospans implicitly use a \emph{composite} of
lax double functors. Specifically, laxators from the lax monoidal functor $F$
combine with laxators from the lax double functor $\Apex$ to give the
distinctive formula for composing decorations of decorated cospans.

\subsection{Application: Double Category of Process Theories}

An early and recurring theme of applied category theory is the mathematical
modeling of physical or computational processes by monoidal categories, often
with extra structure \cite{baez2010}. To describe a process syntactically, one
can define, say by generators and relations, a small category $\cat{T}$ with the
relevant structure, and then choose a particular morphism $p$ in $\cat{T}$. The
category $\cat{T}$ defines the basic material for the process and the morphism
$p$ specifies the process itself. Regarding the category $\cat{T}$ as a theory
in the sense of the categorical logic, the pair $(\cat{T},p)$ might be called a
\emph{theory of a process}, or \emph{process theory} for short. For example, in
the author's thesis \cite{patterson2020}, a \emph{statistical theory} is defined
to be a small Markov category \cite{fritz2020} equipped with extra linear
algebraic structure, together with a distinguished morphism $p: \theta \to x$
representing the data generating process for a statistical model.

To be more precise, process theories are defined relative to a \define{concrete
  2-category}, by which we mean a 2-category $\bicat{C}$ equipped with a
2-functor $|-|: \bicat{C} \to \Cat$, giving the \define{underlying category} of
$\bicat{C}$. This 2-functor will often satisfy additional properties, such as
being locally faithful, but we need not assume that. Given a morphism
$F: \cat{X} \to \cat{Y}$ in a concrete 2-category, we will write
$F(x) := |F|(x)$ and $F(f) := |F|(f)$ for the action of the underlying functor
of $F$ on the objects and morphisms of $|\cat{X}|$. As an example, statistical
theories are based on the concrete 2-category of small linear algebraic Markov
categories, structure-preserving monoidal functors, and monoidal natural
transformations \cite{patterson2020}.

Process theories can be composed once their underlying theories are made open.
In the context of statistics, this composition corresponds to making
hierarchical statistical models, where samples from one model become parameters
of the next. To express this mathematically, we construct a double category of
process theories. We need two main ingredients: the double Grothendieck
construction, and an extension of the familiar construction of comma categories
to a lax double functor. We now review the latter, which is interesting in its
own right.

There is a lax double functor $\Comma: \Csp{\Cat} \to \Span{\Cat}$ that is the
identity on objects and arrows and sends a cospan of categories
$(\cat{A} \xrightarrow{i} \cat{X} \xleftarrow{o} \cat{B})$ to the span of
categories
$(\cat{A} \xleftarrow{\pi_{\cat{A}}} i/o \xrightarrow{\pi_{\cat{B}}} \cat{B})$
comprising the comma category $i/o$ with its canonical projections.\footnote{The
  lax double functor $\Comma: \Csp{\bicat{C}} \to \Span{\bicat{C}}$ even
  generalizes from $\bicat{C} = \Cat$ to any 2-category $\bicat{C}$ with comma
  objects, pushouts, and pullbacks \cite[\S 4.5.9]{grandis2019}, although we
  will not use that.} It acts on maps of cospans as
\begin{equation*}
  \begin{tikzcd}
    {\cat{A}} & {\cat{X}} & {\cat{B}} \\
    {\cat{A}'} & {\cat{X}'} & {\cat{B}'}
    \arrow["i", from=1-1, to=1-2]
    \arrow["o"', from=1-3, to=1-2]
    \arrow["H"', from=1-1, to=2-1]
    \arrow["F"', from=1-2, to=2-2]
    \arrow["K", from=1-3, to=2-3]
    \arrow["{i'}"', from=2-1, to=2-2]
    \arrow["{o'}", from=2-3, to=2-2]
  \end{tikzcd}
  \qquad\mapsto\qquad
  \begin{tikzcd}
    {\cat{A}} & {i/o} & {\cat{B}} \\
    {\cat{A}'} & {i'/o'} & {\cat{B}'}
    \arrow["{\pi_{\cat{A}}}"', from=1-2, to=1-1]
    \arrow["{\pi_{\cat{B}}}", from=1-2, to=1-3]
    \arrow["H"', from=1-1, to=2-1]
    \arrow["{\tilde F}"', from=1-2, to=2-2]
    \arrow["K", from=1-3, to=2-3]
    \arrow["{\pi_{\cat{A}'}}", from=2-2, to=2-1]
    \arrow["{\pi_{\cat{B}'}}"', from=2-2, to=2-3]
  \end{tikzcd},
\end{equation*}
where the functor denoted $\tilde F$ sends an object
$(a,\; i(a) \xrightarrow{f} o(b),\; b)$ of the comma category $i/o$ to
\begin{equation*}
  (H(a),\; i'(H(a)) = F(i(a)) \xrightarrow{F(f)} F(o(b)) = o'(H(b)),\; H(b))
\end{equation*}
and a morphism $(h,k)$ to $(H(h), K(k))$.

To describe the laxators, let
$m = {(\cat{A} \xrightarrow{i} \cat{X} \xleftarrow{o} \cat{B})}$ and
$n = {(\cat{B} \xrightarrow{j} \cat{Y} \xleftarrow{p} \cat{C})}$ be composable
cospans of categories and let
$\iota_{\cat{X}}: \cat{X} \to \cat{X} +_{\cat{B}} \cat{Y}$ and
$\iota_{\cat{Y}}: \cat{Y} \to \cat{X} +_{\cat{B}} \cat{Y}$ be the inclusions
into the pushout of categories. Then the apex map of the laxator $\Comma_{m,n}$
is the functor
\begin{equation*}
  (i/o) \times_{\cat{B}} (j/p) \to (\iota_{\cat{X}} \circ i)/(\iota_{\cat{Y}} \circ p)
\end{equation*}
that sends a pair of objects $(a,f,b)$ and $(b,g,c)$ with $o(b) = j(b)$ to
$(a, \iota_{\cat{Y}}(g) \circ \iota_{\cat{X}}(f), c)$, which is well-defined
since $\iota_{\cat{X}}(o(b)) = \iota_{\cat{Y}}(j(b))$ in
$\cat{X} +_{\cat{B}} \cat{Y}$. This functor sends a pair of maps $(h,k)$ and
$(k,\ell)$ to the map $(h,\ell)$. Finally, given a category $\cat{A}$, the apex
map of the unitor $\Comma_{\cat{A}}$ is the functor
$\cat{A} \to 1_{\cat{A}}/1_{\cat{A}}$ that sends an object $a \in \cat{A}$ to
$(a, 1_a, a)$ and a morphism $h$ to $(h,h)$.

\begin{proposition}
  Let $\bicat{C}$ be a concrete 2-category with pushouts. Then there is a double
  category that has
  \begin{itemize}
    \item as objects, an object $\cat{A}$ in $\bicat{C}$ together with an object
      $a \in |\cat{A}|$;
    \item as arrows $(\cat{A},a) \to (\cat{A}',a')$, a morphism
      $H: \cat{A} \to \cat{A}'$ in $\bicat{C}$ together with a morphism
      $h': H(a) \to a'$ in $|\cat{A}'|$;
    \item as proarrows $(\cat{A},a) \proTo (\cat{B},b)$, a cospan in $\bicat{C}$
      of form $m = (\cat{A} \xrightarrow{i} \cat{X} \xleftarrow{o} \cat{B})$
      along with a morphism $f: i(a) \to o(b)$ in $|\cat{X}|$;
    \item as cells
      $\inlineCell{(\cat{A},a)}{(\cat{B},b)}{(\cat{A}',a')}{(\cat{B}',b')}{(m,f)}{(m',f')}{(H,h')}{(K,k')}{}$,
      a morphism $F: \cat{X} \to \cat{X}'$ forming a map of
      cospans
      $ %
        \begin{tikzcd}[cramped, sep=small]
          {\cat{A}} & {\cat{X}} & {\cat{B}} \\
          {\cat{A}'} & {\cat{X}'} & {\cat{B}'}
          \arrow["i", from=1-1, to=1-2]
          \arrow["o"', from=1-3, to=1-2]
          \arrow["H"', from=1-1, to=2-1]
          \arrow["F"', from=1-2, to=2-2]
          \arrow["K", from=1-3, to=2-3]
          \arrow["{i'}"', from=2-1, to=2-2]
          \arrow["{o'}", from=2-3, to=2-2]
        \end{tikzcd} $
      in $\bicat{C}$ and making the following square in $|\cat{X}'|$ commute:
      \begin{equation*}
        \begin{tikzcd}[column sep=small]
          {i'(H(a))} & {F(i(a))} & {F(o(b))} & {o'(K(b))} \\
          {i'(a')} &&& {o'(b')}
          \arrow["{F(f)}", from=1-2, to=1-3]
          \arrow["{f'}", from=2-1, to=2-4]
          \arrow[Rightarrow, no head, from=1-1, to=1-2]
          \arrow["{i'(h')}"', from=1-1, to=2-1]
          \arrow[Rightarrow, no head, from=1-3, to=1-4]
          \arrow["{o'(k')}", from=1-4, to=2-4]
        \end{tikzcd}.
      \end{equation*}
  \end{itemize}
  Two proarrows
  $(\cat{A},a) \xproTo{(m,f)} (\cat{B},b) \xproTo{(n,g)} (\cat{C},c)$, with
  $m = (\cat{A} \xrightarrow{i} \cat{X} \xleftarrow{o} \cat{B})$ and
  $n = (\cat{B} \xrightarrow{j} \cat{Y} \xleftarrow{p} \cat{C})$, have composite
  $(m \odot n, h): (\cat{A},a) \proTo (\cat{C},c)$, where $m \odot n$ is the
  composite cospan in $\bicat{C}$ with apex $\cat{X} +_{\cat{B}} \cat{Y}$ and
  $h$ is given by first composing the images of $f$ and $g$ in
  ${|\cat{X}| +_{|\cat{B}|} |\cat{Y}|}$ and then applying the canonical functor
  ${|\cat{X}| +_{|\cat{B}|} |\cat{Y}| \to |\cat{X} +_{\cat{B}} \cat{Y}|}$. The
  identity proarrow at $(\cat{A},a)$ is $(\id_{\cat{A}}, 1_a)$.
\end{proposition}
\begin{proof}
  Apply the double Grothendieck construction to the composite lax double functor
  \begin{equation*}
    \Csp{\bicat{C}}
      \xrightarrow{\Csp{|-|}} \Csp{\Cat}
      \xrightarrow{\Comma} \Span{\Cat}.
  \end{equation*}
  Here the lax double functor $\Csp{|-|}$ is a particular case of
  \cref{ex:maps-cospan-dbl-cat}.
\end{proof}

\section{Conclusion}
\label{sec:conclusion}

We have revisited structured and decorated cospans from the perspective of
double category theory, showing that double categories of structured cospans
form cocartesian equipments and that their maps are cocartesian double functors.
We have also reconstructed and generalized double categories of decorated
cospans using the double Grothendieck construction.

Looking to future developments, we have presented a reasonably complete and
self-contained treatment of the theory of structured cospans, but less so for
the theory of decorated cospans. We have not shown how to construct maps between
double categories of decorated cospans, along the lines of Baez et al's
\cite[Theorem 2.5]{baez2022}. Just as the classical Grothendieck construction
for categories is 2-functorial \mbox{\cite[\S 6]{peschke2020}}, so should be the
Grothendieck construction for double categories, which should in turn directly
produce maps between decorated cospan double categories and natural
transformations between those. Equally importantly, we have not recovered the
symmetric monoidal product of decorated cospans, an absence clearly felt in our
example of the double category of process theories. Monoidal products should be
obtained as a corollary of a hypothetical Grothendieck construction for monoidal
double categories, combining the monoidal and double Grothendieck constructions
\cite{moeller2020,cruttwell2022}. In these and other ways, we expect the further
development of the theory of double categories to immediately impact the study
of open systems, simplifying known constructions, suggesting new ones, and
enabling practitioners to focus on applications rather than general theoretical
issues.

\paragraph{Acknowledgments}

This project was partially supported by the Air Force Office of Scientific
Research (AFOSR) Young Investigator Program (YIP) through Award
FA9550-23-1-0133. I thank Nathanael Arkor, John Baez, and Brandon Shapiro for
helpful conversations. I am also grateful to Brandon Shapiro for comments on
early versions of this work.

\newpage
\printbibliography[heading=bibintoc]

\appendix

\section{Cocartesian Equipments}
\label{app:cocartesian-equipments}

In this appendix, we review cocartesian double categories and equipments, and
the maps between them. This material is known but may not be straightforward to
access in the literature. It is included for the reader's convenience.

Just as a cocartesian category is (on one standard definition) a category with
finite coproducts, a cocartesian double category is a double category with
finite double-categorical coproducts. A highly conceptual way to make this
precise is to define a cocartesian double category to be a cocartesian object in
the 2-category $\Dbl$ of double categories, double functors, and natural
transformations. Thus, a double category $\dbl{D}$ is \define{cocartesian} if
the diagonal and terminal double functors,
$\Delta_{\dbl{D}}: \dbl{D} \to \dbl{D} \times \dbl{D}$ and
$!_{\dbl{D}}: \dbl{D} \to \dbl{1}$, have left adjoints in $\Dbl$. This is (dual
to) the approach taken by Aleiferi in her PhD thesis on cartesian double
categories \cite{aleiferi2018}. It will be convenient for us to have a more
concrete description.\footnote{The equivalence of the two definitions follows
  from a general result about double adjunctions \cite[Corollary
  4.3.7]{grandis2019}.}

\begin{definition} \label{def:cocartesian-dbl-cat}
  A double category $\dbl{D}$ is \define{cocartesian} if its underlying
  categories $\dbl{D}_0$ and $\dbl{D}_1$ have finite coproducts; the source and
  target functors $\src, \tgt: \dbl{D}_1 \rightrightarrows \dbl{D}_0$ preserve
  finite coproducts; and the external composition
  $\odot: \dbl{D}_1 \times_{\dbl{D}_0} \dbl{D}_1 \to \dbl{D}_1$ and unit
  $\id: \dbl{D}_0 \to \dbl{D}_1$ also preserve finite coproducts, meaning that
  for all proarrows $x \xproTo{m} y \xproTo{n} z$ and
  $x' \xproTo{m'} y' \xproTo{n'} z'$ and objects $x$ and $x'$ in $\dbl{D}$, the
  canonical comparison cells
  \begin{equation*}
    \begin{tikzcd}
      {x+x'} &&& {z+z'} \\
      {x+x'} &&& {z+z'}
      \arrow[""{name=0, anchor=center, inner sep=0}, "{(m \odot n) + (m' \odot n')}", "\shortmid"{marking}, from=1-1, to=1-4]
      \arrow[Rightarrow, no head, from=1-1, to=2-1]
      \arrow[Rightarrow, no head, from=1-4, to=2-4]
      \arrow[""{name=1, anchor=center, inner sep=0}, "{(m+m') \odot (n+n')}"', "\shortmid"{marking}, from=2-1, to=2-4]
      \arrow["{[\iota_m \odot \iota_n, \iota_{m'} \odot \iota_{n'}]}"{description}, draw=none, from=0, to=1]
    \end{tikzcd}
    \qquad\text{and}\qquad
    \begin{tikzcd}
      {x+x'} && {x+x'} \\
      {x+x'} && {x+x'}
      \arrow[""{name=0, anchor=center, inner sep=0}, "{\mathrm{id}_x + \mathrm{id}_{x'}}", "\shortmid"{marking}, from=1-1, to=1-3]
      \arrow[""{name=1, anchor=center, inner sep=0}, "{\mathrm{id}_{x+x'}}"', "\shortmid"{marking}, from=2-1, to=2-3]
      \arrow[Rightarrow, no head, from=1-1, to=2-1]
      \arrow[Rightarrow, no head, from=1-3, to=2-3]
      \arrow["{[\mathrm{id}_{\iota_x}, \mathrm{id}_{\iota_{x'}}]}"{description}, draw=none, from=0, to=1]
    \end{tikzcd}
  \end{equation*}
  given by the universal property of binary coproducts, as well as the
  comparison cells
  $0_{\dbl{D}_1} \xrightarrow{!} 0_{\dbl{D}_1} \odot 0_{\dbl{D}_1}$ and
  $0_{\dbl{D}_1} \xrightarrow{!} \id_{0_{\dbl{D}_0}}$ given by the universal
  property of initial objects, are all isomorphisms in $\dbl{D}_1$.
\end{definition}

An equipment, also known as a fibrant double category or a framed bicategory, is
a double category in which proarrows can be restricted or extended along pairs
of arrows in a universal way. Equipments can be defined in at least three
equivalent ways \cite[Theorem 4.1]{shulman2008}, including as follows.

\begin{definition} \label{def:equipment}
  An \define{equipment} is a double category $\dbl{D}$ such that the pairing of
  the source and target functors,
  $\langle s,t \rangle: \dbl{D}_1 \to \dbl{D}_0 \times \dbl{D}_0$, is a
  fibration.
\end{definition}

Elaborating the definition, a double category $\dbl{D}$ is an equipment if every
niche in $\dbl{D}$ of the form on the left can be completed to a cell as on the
right
\begin{equation*}
  \begin{tikzcd}
    x & y \\
    w & z
    \arrow["f"', from=1-1, to=2-1]
    \arrow["g", from=1-2, to=2-2]
    \arrow["n"', "\shortmid"{marking}, from=2-1, to=2-2]
  \end{tikzcd}
  \qquad\leadsto\qquad
  \begin{tikzcd}
    x & y \\
    w & z
    \arrow["f"', from=1-1, to=2-1]
    \arrow["g", from=1-2, to=2-2]
    \arrow[""{name=0, anchor=center, inner sep=0}, "n"', "\shortmid"{marking}, from=2-1, to=2-2]
    \arrow[""{name=1, anchor=center, inner sep=0}, "{\mathrm{res}^f_g(n)}", "\shortmid"{marking}, from=1-1, to=1-2]
    \arrow["{\mathrm{res}}"{description}, draw=none, from=1, to=0]
  \end{tikzcd}
\end{equation*}
called a \define{restriction} cell, with the universal property that for every
pair of arrows ${h: x' \to x}$ and ${k: y' \to y}$, each cell $\alpha$ of the
form on the left factors uniquely through the restriction cell as on the
right:
\begin{equation*}
  \begin{tikzcd}
    {x'} & {y'} \\
    x & y \\
    w & z
    \arrow[""{name=0, anchor=center, inner sep=0}, "n"', "\shortmid"{marking}, from=3-1, to=3-2]
    \arrow["f"', from=2-1, to=3-1]
    \arrow["g", from=2-2, to=3-2]
    \arrow["h"', from=1-1, to=2-1]
    \arrow["k", from=1-2, to=2-2]
    \arrow[""{name=1, anchor=center, inner sep=0}, "{m'}", "\shortmid"{marking}, from=1-1, to=1-2]
    \arrow["\alpha"{description}, draw=none, from=1, to=0]
  \end{tikzcd}
  \qquad=\qquad
  \begin{tikzcd}
    {x'} & {y'} \\
    x & y \\
    w & z
    \arrow["f"', from=2-1, to=3-1]
    \arrow["g", from=2-2, to=3-2]
    \arrow[""{name=0, anchor=center, inner sep=0}, "n"', "\shortmid"{marking}, from=3-1, to=3-2]
    \arrow[""{name=1, anchor=center, inner sep=0}, "{\mathrm{res}^f_g(n)}", "\shortmid"{marking}, from=2-1, to=2-2]
    \arrow["h"', from=1-1, to=2-1]
    \arrow["k", from=1-2, to=2-2]
    \arrow[""{name=2, anchor=center, inner sep=0}, "{m'}", "\shortmid"{marking}, from=1-1, to=1-2]
    \arrow["{\mathrm{res}}"{description}, draw=none, from=1, to=0]
    \arrow["{\exists !}"{description, pos=0.4}, draw=none, from=2, to=1]
  \end{tikzcd}
\end{equation*}

Finally, a \define{cocartesian equipment} is a double category that is both
cocartesian and an equipment. We emphasize again that being a cocartesian
equipment is a property of, not a structure on, a double category.

\begin{example}[Cospan double categories] \label{ex:cospan-dbl-cat}
  The prototypical example of a cocartesian equipment is none other than
  $\Csp{\cat{S}}$, the double category of cospans in a category $\cat{S}$ with
  finite colimits. Let us sketch the proof behind this statement. For a more
  detailed proof, one can dualize the proof in Aleiferi's thesis that
  $\Span{\cat{S}}$, for a category $\cat{S}$ with finite limits, is a cartesian
  equipment \cite{aleiferi2018}.

  Finite coproducts in the category $\Csp{\cat{S}}_0 = \cat{S}$ exist by
  assumption, and finite coproducts in the functor category
  $\Csp{\cat{S}}_1 = \cat{S}^{\{\bullet \rightarrow \bullet \leftarrow \bullet\}}$
  are computed pointwise in $\cat{S}$. So the source and target functors
  $\leftfoot, \rightfoot: \Csp{\cat{S}}_1 \to \cat{S}$, extracting the left and
  right feet, preserve coproducts. The comparison cells are isomorphisms because
  colimits commute with colimits (specifically, pushouts commute with
  coproducts) up to canonical isomorphism. Thus, the double category of cospans
  is cocartesian.

  It is also an equipment. To restrict a cospan
  $c \xrightarrow{\ell} y \xleftarrow{r} d$ along a pair of morphisms
  $f: a \to c$ and $g: b \to d$, simply compose the morphisms with the legs of
  the cospan. The restriction cell is trivial:
  \begin{equation*}
    \begin{tikzcd}
      a & y & b \\
      c & y & d
      \arrow["\ell"', from=2-1, to=2-2]
      \arrow["r", from=2-3, to=2-2]
      \arrow[Rightarrow, no head, from=1-2, to=2-2]
      \arrow["f"', from=1-1, to=2-1]
      \arrow["{\ell \circ f}", from=1-1, to=1-2]
      \arrow["{r \circ g}"', from=1-3, to=1-2]
      \arrow["g", from=1-3, to=2-3]
    \end{tikzcd}.
  \end{equation*}
\end{example}

We turn now to maps between cocartesian double categories and equipments. Since
a cocartesian category is a cocartesian object in $\Dbl$, a map between
cocartesian double categories can be defined abstractly as a cocartesian
morphism between cocartesian objects \cite[\S 5.2]{carboni1991}. As before, this
definition reduces to a more concrete one:

\begin{definition} \label{def:cocartesian-dbl-functor}
  A double functor $F: \dbl{D} \to \dbl{E}$ between cocartesian double
  categories is \define{cocartesian} if both underlying functors
  $F_0: \dbl{D}_0 \to \dbl{E}_0$ and $F_1: \dbl{D}_1 \to \dbl{E}_1$ preserve
  finite coproducts.
\end{definition}

Note that we will apply this definition to lax as well as pseudo double
functors.

Perhaps surprisingly, no extra conditions on double functors between equipments
are required. Any (op)lax double functor between equipments automatically
preserves restriction (respectively, extension) cells, as proved by Shulman
\cite[Proposition 6.4]{shulman2008}. In particular, a pseudo double functor
between equipments preserves all the operations afforded by an equipment.

\begin{example}[Maps between cospan double categories]
  \label{ex:maps-cospan-dbl-cat}
  The construction of the double category of cospans $\Csp{\cat{S}}$ extends to
  a 2-functor $\Csp: \Cat_{\mathrm{po}} \to \DblLax$, where $\Cat_{\mathrm{po}}$
  is the 2-category of categories with chosen pushouts, arbitrary functors, and
  natural transformations and $\DblLax$ is the 2-category of double categories,
  lax double functors, and natural transformations \cite[\S C3.11]{grandis2019}.

  Let us describe the lax double functor
  $\Csp{F}: \Csp{\cat{S}} \to \Csp{\cat{S}'}$ induced by a functor
  $F: \cat{S} \to \cat{S}'$ between categories with pushouts. We have
  $\Csp{F}_0 = F$ on objects and arrows, while $\Csp{F}_1$ post-composes with
  $F$ the diagrams defining cospans and maps of cospans in $\cat{S}$. Since
  functors preserve identities, $\Csp{F}$ is a \emph{normal} lax double functor,
  meaning that it preserves identity proarrows strictly. Given cospans
  $m = (a \rightarrow x \leftarrow b)$ and $n = (b \rightarrow y \leftarrow c)$
  in $\cat{S}$, the laxator
  \begin{equation*}
    \Csp{F}_{m,n}: \Csp{F}(m) \odot \Csp{F}(n) \to \Csp{F}(m \odot n)
  \end{equation*}
  has apex map given by the universal property of the pushout in $\cat{S}'$:
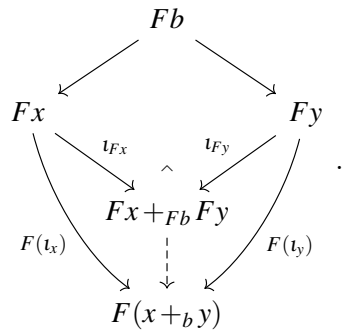
  \begin{equation*}
    \begin{tikzcd}[column sep=small]
      & Fb \\
      Fx && Fy \\
      & {Fx +_{Fb} Fy} \\
      & {F(x +_b y)}
      \arrow[from=1-2, to=2-1]
      \arrow[from=1-2, to=2-3]
      \arrow["{\iota_{Fx}}", from=2-1, to=3-2]
      \arrow["{\iota_{Fy}}"', from=2-3, to=3-2]
      \arrow["\lrcorner"{anchor=center, pos=0.125, rotate=135}, draw=none, from=3-2, to=1-2]
      \arrow["{F(\iota_x)}"', curve={height=12pt}, from=2-1, to=4-2]
      \arrow[dashed, from=3-2, to=4-2]
      \arrow["{F(\iota_y)}", curve={height=-12pt}, from=2-3, to=4-2]
    \end{tikzcd}.
  \end{equation*}
  Clearly, $\Csp{F}$ is pseudo if and only if $F$ preserves pushouts.

  Suppose that $\cat{S}$ and $\cat{S}'$ have all finite colimits, so that their
  double categories of cospans are cocartesian. Since coproducts of cospans are
  computed pointwise, $\Csp{F}$ is a cocartesian lax double functor exactly when
  $F$ preserves finite coproducts. Altogether, $\Csp{F}$ is a cocartesian
  (pseudo) double functor if and only if $F$ preserves all finite colimits.
\end{example}

\end{document}